\documentclass[11pt]{amsart}

\usepackage{enumerate}
\usepackage{mathrsfs}
\usepackage{amssymb,amsmath,hyperref}
\usepackage[all]{xy}

\title[]{Newton--Okounkov bodies and complexity functions}

\thanks{DS was supported by DFG research fellowship SCHM 2332/1-1}


\author{Mihai Fulger}
\address{Institute of Mathematics of the Romanian Academy, P. O. Box 1-764, RO-014700,
Bucharest, Romania}
\curraddr{Department of Mathematics,
	Princeton University,
	Princeton, NJ 08544, USA}
\email{afulger@princeton.edu}

\author{David Schmitz}
\address{Department of Mathematics, 
   Stony Brook University, 
   Stony Brook, 
   NY 11794, USA}
\email{david.schmitz@stonybrook.edu}

\renewcommand\ge{\geqslant}  
\renewcommand\le{\leqslant}  
\renewcommand\geq{\geqslant}  
\renewcommand\leq{\leqslant}  
\renewcommand\epsilon{\varepsilon}
\renewcommand\phi{\varphi}

\renewcommand\subset{\subseteq}
\renewcommand\tilde{\widetilde}
\renewcommand\O{{\mathcal O}}
\renewcommand\P{\mathbb P}

\renewcommand{\tilde}{\widetilde}

\newcommand\newop[2]{\newcommand#1{\mathop{\rm #2}\nolimits}}

\newop{\dv}{div}
\newop{\id}{id}
\newop{\ord}{ord}
\newop{\Div}{Div}
\newop{\Ext}{Ext}
\newop{\Face}{Face}
\newop{\Vol}{Vol}
\newop{\Neg}{Neg}
\newop{\Cox}{Cox}
\newop{\Null}{Null}
\newop{\Nef}{Nef}
\newop{\supp}{supp}
\newop{\NE}{\overline{{NE}}}
\newop{\Eff}{\overline{{Eff}}}
\newop{\valnu}{\nu_{Y_\bullet}}
\newop\NS{NS} 
\newop\vol{vol}
\newop\BigCone{Big}

\newtheorem{prop}{Proposition}
\newtheorem{lemma}[prop]{Lemma}

\newtheorem{thm}[prop]{Theorem}

\theoremstyle{definition}

\newtheorem{example}[prop]{Example}
\newtheorem{remark}[prop]{Remark}
\newtheorem*{ques}{Question}

\begin{document}

\begin{abstract}
We show that quite universally the holonomicity of the complexity function of a divisor does not predict whether the Newton--Okounkov body is polyhedral for every choice of a flag. 
\end{abstract}

\maketitle

Let $X$ be a complex projective variety of dimension $d$ with a complete flag of subvarieties $Y_{\bullet}:\ X=Y_0\supset\ldots \supset Y_d$ such that $\dim Y_i=d-i$ and $Y_d$ is a smooth point for all $Y_i$.
For any big divisor $D$ on $X$, \cite{LM} and \cite{KK} construct a convex body $\Delta_{Y_{\bullet}}(X;D)\subset\mathbb R^d$. This Newton--Okounkov body encodes asymptotic information about the sections of multiples of $D$. 
For example, its Euclidean volume and the volume of $D$ as
a divisor on $X$ (cf. \cite[\S2.2.C]{laz}) coincide up to a normalization factor.

Work of \cite{LM} and \cite{J} suggests that the Newton--Okounkov bodies encode many 
important invariants of the divisor $D$. For example \cite{J} shows that the set of bodies 
$\Delta_{Y_{\bullet}}(X;D)$ considering \emph{all} flags $Y_{\bullet}$ as above recovers the numerical
class of $D$. More recently, \cite{KL1,KL2} recover the non-ample locus ${\mathbb B}_+(D)$
and non-nef locus ${\mathbb B}_-(D)$ of $D$ from the bodies associated to certain flags.  
On the other hand, \cite{klm} show that the particular shape of $\Delta_{Y_{\bullet}}(X;D)$
for many flags $Y_{\bullet}$ is not necessarily descriptive of the positivity properties of $D$.

In this note we further this investigation on algebro-geometric consequences of the shape of Newton--Okounkov bodies. We show that quite universally a certain regularity condition on the cohomology of
$D$ that has been considered by Katzarkov and Liu \cite{kl} among others does not determine the polyhedrality of $\Delta_{Y_{\bullet}}(X;D)$. The strategy is to lift a generalization of an example of \cite{klm} to some
representative of each birational class of varieties of dimension $d\geq 4$.

The complexity function of a functor $F$ on a category $\mathcal C$ with respect to objects $E_1,E_2$ in $\mathcal C$ is defined in \cite{ekk,kl} by
$$
	E_{E_1,E_2,F}(x,q)=\sum_{n,i} \dim\Ext^i\left(E_1,F^nE_2\right)x^nq^i.
$$
In particular, they ask whether $E_{E_1,E_2,F}$ is holonomic. This would imply that the function 
$E_{E_1,E_2,F}(x,q)$ is determined by a finite set of data. Concretely, a formal power series $f=\sum_{i,n}a_{i,n}x^nq^i$ is \emph{holonomic}, or \emph{D-finite} (see \cite[Definition 15]{z}), if there exist tuples of not all zero polynomials $p_0,\dots,p_k$ and $g_0,\dots,g_k$ in $x$ and $q$ such that $f$ satisfies linear differential equations
\begin{eqnarray*}
	p_0\frac{\partial^kf}{\partial x^k}+p_1\frac{\partial^{k-1}f}{\partial x^{k-1}}+\dots+p_kf &=& 0\\
g_0\frac{\partial^kf}{\partial q^k}+g_1\frac{\partial^{k-1}f}{\partial q^{k-1}}+\dots+g_kf &=& 0.
\end{eqnarray*}
\noindent In general, $E_{E_1,E_2,F}(x,q)$ cannot be expected to be holonomic. 
\cite{ekk} gives an example where $X$ is a $\mathbb P^1$-bundle over a product of elliptic curves, $E_1=E_2=\mathcal O_X$, and $F$ is the twist by an appropriately chosen big divisor. 
The authors suggest that this failure of being holonomic can be linked to the non-polyhedrality of the categorical Okounkov-body of $D$ (cf. \cite[Section 4]{kl}). We consider the following 

\begin{ques}Let $X=Y_0\supset Y_1\supset\ldots\supset Y_d$ be an admissible flag
on the projective variety $X$ of dimension $d$, i.e., $Y_k$ is irreducible
of codimension $k$ and smooth at the point $Y_d$ for all $k$.
Let $D$ be a big divisor on $X$ whose associated Newton--Okounkov body 
$\Delta_{Y_{\bullet}}(X;D)$ is nonpolyhedral.
Is the complexity function 
$$E_{X,D}(x,q):=\sum_{n,i}\dim H^i\left(X;\mathcal O_X(nD)\right)x^nq^i$$ 
necessarily non-holonomic?
\end{ques}

Our first counterexample has $D$ \emph{ample}:

\begin{example}[\cite{klm}]\label{ex:klm}
Let $X=\mathbb P^2\times\mathbb P^2$. Consider $E\subset\mathbb P^2$ an elliptic curve
without complex multiplication. Construct the flag:
\begin{itemize}
\item $Y_0=X$
\item $Y_1=\mathbb P^2\times E$ 
\item $Y_2=E\times E$ 
\item $Y_3$ is general in the complete linear series $|f_1+f_2+\Delta_E|$ on $Y_2$.
Here $f_1,f_2$ are the fibers of the projections to the factors in the product, and $\Delta$ is the diagonal.
\item $Y_4$ is a general point on $Y_3$.
\end{itemize}
Let $D$ be a divisor with associated line bundle $\mathcal O(3,1)$.
\cite[Example 3.4]{klm} shows that the associated Newton--Okounkov body $\Delta_{Y_{\bullet}}(X;D)$
is not polyhedral. 

On the other hand, the complexity function of an ample divisor is holonomic by the following lemma.
\qed
\end{example}

\begin{lemma}\label{l:semiample} Let $X$ be a projective variety of dimension $d$ and let $D$ be a semi-ample divisor on it.
Then the complexity function $E_{X,D}(x,q)$ is holonomic.
\end{lemma}
\begin{proof}Assume first that $\mathcal O_X(D)$ is globally generated.
Let $\pi:X\to\mathbb P^N$ be the morphism determined by the 
complete linear series $|D|$. Then $D=\pi^*\mathcal O_{\mathbb P^N}(1)$.
The ampleness of the hyperplane class on $\mathbb P^N$, the Leray spectral sequence,
and Serre vanishing imply 
$$H^i(X;\mathcal O_X(nD))=H^0\left(\mathbb P^N;R^i\pi_*\mathcal O_X\otimes\mathcal O_{\mathbb P^N}(n)\right)$$
for $n$ sufficiently large. There exists polynomials $P_i$ such that $P_i(n)=\dim H^0(\mathbb P^N;R^i\pi_*\mathcal O_X\otimes\mathcal O_{\mathbb P^N}(n))$ for any $i$ and sufficiently large $n$.
Then up to finitely many terms, which in any case do not influence holonomicity (cf. \cite[Proposition 2.3.(ii)]{lip}),
$$E_{X,D}(x,q)=\sum_{i=0}^{d}q^i\sum_r P_i(r)x^r.$$
This is a rational function in $x$, polynomial in $q$. 
As such it is algebraic, and in particular holonomic (cf. \cite[Proposition 2.3]{lip}).

When $D$ is only semi-ample, let $m$ be such that $|mD|$ is basepoint free.
Let $\pi:X\to\mathbb P^N$ be the morphism determined by this linear series so that
$\mathcal O_X(mD)=\pi^*\mathcal O_{\mathbb P^N}(1)$.
Write $n=am+r$ with $0\leq r<m$.
Then $$H^i(X;\mathcal O_X(nD))=H^0\left(\mathbb P^N;R^i\pi_*\mathcal O_X(rD)\otimes\mathcal O_{\mathbb P^N}(a)\right).$$
For large $n$, as in the globally generated case, its dimension is $P_{i,r}(a)$ for some polynomial $P_{i,r}$. It is then easy to see that the complexity function is again algebraic. 
\end{proof}

We have so far seen that the answer to the question above is negative by giving one example. In the remainder of this note we will see that in a sense it is universally negative. Concretely, we show that in each birational equivalence class of varieties of dimension $d\ge4$ there exists a smooth model $X$ carrying an admissible flag and a big and semi-ample divisor such that the corresponding Newton--Okounkov body is non-polyhedral, while by Lemma \ref{l:semiample} the complexity function $E_{X,D}(x,q)$ is holonomic.

\begin{thm}
	Let $X$ be a normal projective variety of dimension $d\ge 4$. 
	Then there exists a birational model
	$\tilde X\to X$ containing an admissible flag $X_\bullet$ and a semi-ample
	divisor $H$ such that the Newton--Okounkov body $\Delta_{X_\bullet}(\tilde X;H)$ is non-polyhedral.
\end{thm}
\begin{proof}
Let us first assume that $d=4$.
Up to blowing-up, we may assume that
there exists a generically finite morphism $\pi:X\to\mathbb P^2\times \mathbb P^2$
with $X$ smooth.
(e.g. choose an appropriate birational model of a Noether normalization $X\to\mathbb P^4$).

We may choose the flag elements in the example of \cite{klm} in the previous section 
such that $X_k:=\pi^{-1}Y_k$ is irreducible and smooth for all $k<4$. 
(If $Y_1=\mathbb P^2\times E$ is general in $|\mathcal O(0,3)|$, then its inverse image $X_1$ is
smooth irreducible. To construct $Y_2$, consider the first projection $\mathbb P^2\times E\to\mathbb P^2$ and pullback a general $PGL(3)$ translate of $E\subset\mathbb P^2$. The pullback $X_2$ of $Y_2$
to $X_1$ is then smooth by Kleiman transversality,  
connected by the Fulton--Hansen theorem \cite[Theorem 3.3.6]{laz}, hence also irreducible. See also
\cite[Example 3.3.10]{laz}. The curve $Y_3$ can be chosen general in a very ample linear
series, thus its pullback to $X_2$ is smooth irreducible.)
Let $X_4$ be an arbitrary
point in $\pi^{-1}Y_4$. Then $X_{\bullet}$ is an admissible flag on $X$. Put $H:=\pi^*D$

We claim that for some $\epsilon>0$, the slice $\Delta_{X_{\bullet}}(X;H)\cap\big(\{0\}\times[0,\epsilon]\times\mathbb R^2\big)$ is non-polyhedral.
Pick $\epsilon>0$ such that $D|_{Y_1}-s Y_2$ is ample for all $s\in[0,\epsilon]$. 
It is not hard to check that for $s\in[0,\epsilon]$ the identity
$$\Delta_{X_{\bullet}}(X;H)\cap\big(\{0\}\times\{s\}\times\mathbb R^2\big)=
\Delta_{X_{\bullet}}(X_2;\pi^*(D|_{Y_1}-s Y_2))$$
holds: \cite[Proposition 3.1]{klm} handles the case of restricting ample divisors. The
big and semiample case follows from this by the continuity of slices in the global Newton-Okounkov
body of $X$ by considering a collection of divisors $E_s$ on $X_1$ 
such that $\pi^*(D|_{Y_1}-sY_2)-tE_s$ is ample on $X_1$ for all $t\in[0,1]$ and $s\in(0,\epsilon]$.

As in \cite[Remark 3.3, Example 3.4]{klm}, to conclude that $\Delta_{X_{\bullet}}(X;H)$ is not polyhedral, it is enough to show that the cone which is the translation
by $[H|_{X_2}]$ of the convex span of $-[X_2|_{X_2}]$ and $-[X_3]$ meets the boundary
of $\Nef(X_2)$ along a curve that is not piecewise linear.

For this, observe that $\pi^*\alpha\in\Nef(X_2)$ iff $\alpha\in\Nef(Y_2)$, and $\pi^*\alpha\in\NE(X_2)$ iff $\alpha\in\NE(Y_2)$. Similar equivalences hold for ample and for big classes respectively. Furthermore, $\Nef(Y_2)=\NE(Y_2)$ by \cite[Example 3.4]{klm} and $\pi^*:\NS(Y_2)\to\NS(X_2)$ is a linear injection, since $\pi$ is dominant. 
Therefore the boundary curve (in the sense of the previous paragraph) on $\Nef(X_2)$ is identified via pullback with the boundary curve on $\Nef(Y_2)$. By \cite[Example 3.4]{klm}, the latter is conic, not piecewise linear.

For $X$ of arbitrary dimension $d\ge4$, as above we can assume that there is a generically 
finite dominant morphism $\pi: X\to \P^{d-2}\times\P^2$ with $X$ smooth. Consider in the image the flag $Y_\bullet$ given as follows:
$$
	Y_k = \mathbb P^{d-2-k}\times\P^2
$$
for $k\le d-4$, and $Y_{d-3},\ldots,Y_{d}$ is the flag from Example \ref{ex:klm}.
 We can again argue by Bertini type arguments that the pre-images $X_k:=\pi^{-1}Y_k$ for $k<d$ are smooth and irreducible, thus any choice of the point $X_d$ makes $X_\bullet$ into an admissible flag. Now the same arguments as above yield that for $D=\pi^*\O_{\P^{d-2}\times\P^2}(3,1)$ and some $\epsilon>0$ the slice 
 $$\Delta_{X_{\bullet}}(X;D)\cap\big(\{0\}^{d-3}\times [0,\epsilon]\times\mathbb R^2\big)
 $$
 is non-polyhedral.
\end{proof}

\begin{remark}
Considering the original question, it is natural to ask whether at least the reverse implication is true, i.e., whether the polyhedrality of some Newton--Okounkov body implies holonomicity of the corresponding complexity function.  However, the fact that Newton--Okounkov bodies do not carry information about all sections in each degree of a graded linear series suggests a negative answer to this question as well. A good candidate would be a linear series whose semi-group of valuation vectors is not finitely generated but whose Newton--Okounkov body is polyhedral nonetheless. 
\end{remark}
%
%
%
%
%
%

\section*{Acknowledgments} We thank Robert Lazarsfeld for raising the question and for useful discussions.

\end{document}